\documentclass[a4paper,12pt]{amsart}

\usepackage{setspace,graphicx}
\usepackage{stackrel}                                          
\usepackage{amssymb}
\usepackage{verbatim} 
\usepackage{blindtext}
\usepackage{tabularx}
\usepackage{longtable}
\usepackage{amsmath} 
\usepackage{amsthm} 
\usepackage{mathrsfs} 
\usepackage{amsfonts} 
\usepackage{latexsym} 
\usepackage{amscd} 
\usepackage[latin1]{inputenc} 
\usepackage[english]{babel} 
\usepackage{enumerate} 
\usepackage{afterpage}
\usepackage{tikz-cd}
\usepackage{mathtools}

\allowdisplaybreaks

\newcommand{\FF}{\mathbb{F}}  
\newcommand{\NN}{\mathbb{N}} 
\newcommand{\QQ}{\mathbb{Q}}  

\newcommand{\ZZ}{\mathbb{Z}}

\newtheorem{Theorem}{Theorem}[section]
\newtheorem{Lemma}[Theorem]{Lemma} 
\newtheorem{Conjecture}[Theorem]{Conjecture} 
\newtheorem{Proposition}[Theorem]{Proposition} 
\newtheorem{Corollary}[Theorem]{Corollary} 
\newtheorem{Definition}[Theorem]{Definition}  
\newtheorem{Remark}[Theorem]{Remark} 
\newtheorem{Example}[Theorem]{Example}

\DeclareMathOperator{\ord}{ord}

\begin{document}

\author{Seoyoung Kim}
\author{Alex Walsh}

\email{Seoyoung\_Kim@math.brown.edu}
\address{Mathematics Department, Brown University, Box 1917, 151 Thayer Street, Providence, RI 02912 USA}

\email{alexandra\_walsh@brown.edu}
\address{Box 5949, 69 Brown Street, Providence, RI 02912 USA}

\title[Elliptic Fermat Numbers and Elliptic Divisibility Sequence]{Elliptic Fermat Numbers and Elliptic Divisibility Sequence}

\date{\today}

\begin{abstract}

For a pair $(E,P)$ of an elliptic curve $E/\mathbb{Q}$ and a nontorsion point $P\in E(\QQ)$, the sequence of \emph{elliptic Fermat numbers} is defined by taking quotients of terms in the corresponding elliptic divisibility sequence $(D_{n})_{n\in\NN}$ with index powers of two, i.e. $D_{1}$, $D_{2}/D_{1}$, $D_{4}/D_{2}$, etc. Elliptic Fermat numbers share many properties with the classical Fermat numbers, $F_{k}=2^{2^k}+1$. In the present paper, we show that for magnified elliptic Fermat sequences, only finitely many terms are prime. We also define \emph{generalized elliptic Fermat numbers} by taking quotients of terms in elliptic divisibility sequences that correspond to powers of any integer $m$, and show that many of the classical Fermat properties, including coprimality, order universality and compositeness, still hold.

\end{abstract}



\maketitle
\tableofcontents

\section{Introduction}
\noindent
Let $E$ be an elliptic curve defined over $\QQ$ by a Weierstrass equation with integer coefficients
\begin{equation}
\label{WE}
E: y^{2}+a_{1}y+a_{3}xy=x^{3}+a_{2}x^{2}+a_{4}x+a_{6}.
\end{equation}

We say (\ref{WE}) is {\it minimal} if the discriminant $|\Delta(E)|$ is minimal among all Weierstrass equations for $E$. Moreover, we say a minimal Weierstrass equation is {\it reduced} if $a_{1},a_{3}\in\{0,1\}$ and $a_{2}\in\{-1,0,1\}$. It is worth noting that for all elliptic curves over $\QQ$, minimal models exist and a reduced minimal model is unique.

For a fixed nontorsion point $P\in E(\QQ)$, we can define the elliptic divisibility sequence as follows:

\begin{Definition}
The {\it elliptic divisibility sequence} (EDS) associated to the pair $(E,P)$ is the sequence $D=(D_{n})_{n\in\NN}:\NN\rightarrow \NN$ defined by taking the positive square root of the denominator of successive iterations of $P$ as a lowest fraction, i.e.,
$$[n]P=\left(\frac{A_{n}}{D_{n}^{2}}, \frac{B_{n}}{D_{n}^{3}}\right),$$
where $\gcd(A_{n},D_{n})=\gcd(B_{n},D_{n})=1$. An EDS is {\it minimal} if the Weierstrass equation of $E/\QQ$ is minimal and reduced. Also, an EDS is {\it normalized} if $D_{1}=1$.   
\end{Definition}
In this paper, we always assume $D_{1}=1$. Much of our work revolves around \emph{elliptic Fermat numbers}, analogues of the classical Fermat numbers ($F_{n}=2^{2^{n}}+1$, $n\geq 0$) defined by S. Binegar, R. Dominick, M. Kenney, J. Rouse, and A. Walsh in \cite{BDKRW}. In the original version of the paper, they define elliptic Fermat numbers as follows:

\begin{Definition}
\label{Definition2}
Let $D=(D_{n})_{n\in\NN}$ be an EDS. Define the sequence of elliptic Fermat numbers (EFN) $\{F_{k}(E,P)\}_{k\geq 1}$ as follows:
\begin{displaymath} 
F_{k}(E,P) = \left\{ 
\begin{array}{ll} 
\frac{D_{2^{k}}}{D_{2^{k-1}}} & \textrm{if $k\geq 1$}\\ 
D_{1} & \textrm{k=0}
\end{array} 
\right. \end{displaymath}
\end{Definition}    

Note that in the final version of the paper includes slightly different definition of elliptic Fermat numbers, on the other hand, it does not have much effect on our proof too much, which uses the original definition. In \cite[Theorem 9]{BDKRW}, they list a certain set of conditions which force $F_{k}(E,P)$ to be composite for all $k\geq 1$: 

\begin{Theorem}[BDKRW]
For an elliptic curve $E: y^2 = x^3 + ax^2 + bx + c$, assume the following:
\begin{enumerate}[(i)]
\item $E(\mathbb{Q}) = \langle P,T \rangle$, where $P$ has infinite order and $T$ is a rational point of order $2$.
\item $E$ has an egg.
\item $T$ is on the egg.
\item $T$ is the only integral point on the egg.
\item $P$ is not integral.
\item $\gcd(b,m_0) = 1$.
\item $|\tau_k| = 2$ for all $k$.
\item $2 \nmid e_k$ for all $k$.
\item The equations $x^4 + ax^2y^2 + by^4 = \pm 1$ has no integer solutions where $y \not \in \{0,\pm 1 \}$.
\end{enumerate}
Then $F_k(E,P)$ is composite for all $k \geq 1$.
\end{Theorem}

In the same vein, we prove the non-primality of a sequence of elliptic Fermat numbers which is defined by magnified elliptic divisibility sequences. First, we recall the definition of magnified elliptic divisibility sequences.  

\begin{Definition}
Let $E/\QQ$ and $E'/\QQ$ be two elliptic curves. We say a rational point $P\in E(\QQ)$ is {\it magnified} if $P=\phi(P')$ for some (nonzero) isogeny $\phi: E'\rightarrow E$ over $\QQ$ and some $P'\in E'(\QQ)$. Moreover, an EDS $D=(D_{n})_{n\geq 1}$ is {\it magnified} if $D$ is a minimal EDS associated to some magnified point on an elliptic curve over $\QQ$. We call a sequence of elliptic Fermat numbers $\{F_{k}(E,P)\}_{k\geq 1}$ {\it magnified} if it is defined by using a magnified EDS. 
\end{Definition}

In Section $3$, we prove the following non-primality result of the sequence of magnified elliptic Fermat numbers.

\begin{Theorem}
Let $E/\QQ$ be a minimal magnified elliptic curve with a fixed point $P\in E(\QQ)$ having a (nonzero) odd-degree isogeny $\phi: E'\rightarrow E$ from a minimal elliptic curve $E'/\QQ$ satisfying $\phi(P')=P$. Then the terms $F_{k}(E,P)$ are composite for sufficiently large $k$.
\end{Theorem}

In section $4$, we consider a generalization of elliptic Fermat numbers. Generalized classical Fermat numbers have the form $a^{2^n} + b^{2^n}$ for some relatively prime integers $a$ and $b$. It is natural to consider a similar generalization of elliptic Fermat numbers:

\begin{Definition}
\label{GEFN def}
Let $D=(D_{n})_{n\in\NN}$ be an EDS, and let $m 
\geq 1$ be an integer. We define the sequence of {\it generalized elliptic Fermat numbers} $\{{F_k}^{(m)}(E,P)\}$ as follows:
\begin{displaymath} 
{F_k}^{(m)}(E,P) = \left\{ 
\begin{array}{ll} 
\frac{D_{m^{k}}}{D_{m^{k-1}}} & \textrm{if $k\geq 1$}\\ 
D_{1} & \textrm{if $k = 0.$}
\end{array} 
\right. \end{displaymath}
\end{Definition}

Note that Definition \ref{Definition2} is the special case where $m = 2$. In \cite[Theorem 3, Theorem 4]{BDKRW}, they prove the following theorems about elliptic Fermat numbers:

\begin{Theorem}[BDKRW]
For all $k \neq \ell$, $\gcd(F_k (E, P),F_\ell(E, P))\in \lbrace 1,2 \rbrace$.
\end{Theorem}

\begin{Theorem}[BDKRW]
Let $\Delta(E)$ be the discriminant of $E$ and suppose that $N$ is a positive
integer with $\gcd(N,6 \Delta(E)) = 1$. Then $P$ has order $2^k$ in $E(\mathbb{Z}/N\mathbb{Z})$ if and only if $N \mid F_0(E,P) \cdots F_k(E,P) $ and $ N \nmid F_0(E,P) \cdots F_{k-1}(E,P).$ 
\end{Theorem}

For generalized elliptic Fermat numbers generated by odd $m$, we prove slightly weakened generalizations of these properties, namely:

\begin{Theorem}[Coprimality]
\label{GEFN coprimality theorem}
Let $F=(F_{k}^{(m)}(E,P))_{k\in\NN}$ be the sequence of generalized elliptic Fermat numbers for a fixed elliptic curve $E$, a rational point $P\in E(\QQ)$ and an odd integer $m \geq 1$. Then for all distinct $k, \ell \geq 0$, 
$$\gcd({F_k}^{(m)}(E,P),{F_\ell}^{(m)}(E,P)) \mid m.$$
\end{Theorem}

\begin{Theorem}[Order Universality]
\label{GEFN OU theorem}
Let $F=(F_{k}^{(m)})_{k\in\NN}$ be the sequence of generalized elliptic Fermat numbers for a fixed elliptic curve $E$, a rational point $P\in E(\QQ)$ and an odd integer $m \geq 1$. Then for all $N \in \NN$ satisfying $\gcd(N, 6\Delta(E)) = 1$,
$$P \textrm{ has order } m^k \textrm{ in } E(\ZZ/N\ZZ) \Longleftrightarrow N \mid {F_0}^{(m)} \cdots {F_k}^{(m)} \textrm{ and } N \nmid {F_0}^{(m)}\cdots {F_{k-1}}^{(m)}.$$ 
\end{Theorem}

We also have an analogous non-primality result for generalized elliptic Fermat numbers.

\begin{Theorem}
Let $E/\QQ$ be a minimal magnified elliptic curve with a fixed point $P\in E(\QQ)$ having a (nonzero) degree $d$ isogeny $\phi: E'\rightarrow E$ from a minimal elliptic curve $E'/\QQ$ satisfying $\phi(P')=P$. For $m$ relatively prime to $d$, $F^{(m)}_{k}(E,P)$ are composite for sufficiently large $k$.
\end{Theorem}

\subsection{Acknowledgements}
We would like to thank the 2017 Wake Forest REU research group, whose work in defining and proving properties of elliptic Fermat numbers inspired our investigation of generalized elliptic Fermat numbers. We are also grateful to Professor Joe Silverman and Yuwei Zhu for their helpful advising.





\section{Non-primality Conjecture for Elliptic Fermat Numbers}

Everest, Miller, and Stephens \cite{EMS} proved the following conjecture for {\it magnified} elliptic divisibility sequences.

\begin{Conjecture}[Primality Conjecture]
Let $D=(D_{n})_{n\in\NN}$ be an elliptic divisibility sequence. Then $D$ contains only finitely many prime terms.
\end{Conjecture}

Note that for the Fibonacci sequence, it is conjectured that primes occur infinitely many times. In this section, we will prove the following conjecture for magnified elliptic Fermat numbers.

\begin{Conjecture}[Primality Conjecture for elliptic Fermat numbers]
Let $F=(F_{n})_{n\in\NN}$ be the sequence of elliptic Fermat numbers for a fixed elliptic curve $E$ and a rational point $P\in E(\QQ)$. Then $F$ contains only finitely many prime terms.
\end{Conjecture} 

First, we need the following result from \cite{BDKRW}:

\begin{Corollary}
\label{Cor1}
For any odd prime $p\nmid 6\Delta(E)$, $P$ has order $2^{k}$ in $E(\FF_{p})$ if and only if $p\mid F_{k}(E,P)$.
\end{Corollary}

Using the above corollary, we can prove the following result.

\begin{Lemma}
\label{Lemma2}
Let $E/\QQ$ be a minimal magnified elliptic curve with a fixed point $P\in E(\QQ)$ having a (nonzero) odd-degree $d$ isogeny $\phi: E'\rightarrow E$ from a minimal elliptic curve $E'/\QQ$ satisfying $\phi(P')=P$. For sufficiently large $k$, we have
$$\gcd(F_{k}(E',P'),F_{k}(E,P))\neq 1.$$
\end{Lemma}

\begin{proof}
Let $p$ be a fixed prime which divides $F_{k}(E',P')$. Let $S$ be the set of primes for which $E$, $E'$, and $\phi$ cannot give an isogeny modulo $p$. Note that $S$ is a finite set. Without loss of generality, we can assume $p\notin S$. From Corollary \ref{Cor1}, it is sufficient to prove that the isogeny $\phi$ reduction modulo $p$ (which is again an isogeny) preserves the order $2^{k}$ of $P'$
$$\phi:E'(\FF_{p})\rightarrow E(\FF_{p}).$$
We consider the dual isogeny $\hat{\phi}$ of $\phi$. We know
$$\hat{\phi}\circ \phi=[d]$$
where $[d]$ is the multiplication-by-$d$ map on $E'$. Then the map
$$[d]: E'(\FF_{p})\xrightarrow{\phi} E(\FF_{p})\xrightarrow{\hat{\phi}} E'(\FF_{p})$$
preserves the order of $2^{k}$ of $P'$ and so does $\phi$.    
\end{proof}

\begin{Proposition}
\label{Corollary2}
Let $E/\QQ$ be a minimal magnified elliptic curve with a fixed point $P\in E(\QQ)$ having a (nonzero) odd-degree isogeny $\phi: E'\rightarrow E$ from a minimal elliptic curve $E'/\QQ$ satisfying $\phi(P')=P$. For sufficiently large $k$, we actually have the following divisibility:
$$F_{k}(E',P')\mid F_{k}(E,P).$$
\end{Proposition}

\begin{proof}
The proof of Lemma \ref{Lemma2} essentially implies the result.
\end{proof}

\begin{Definition}
Let $E/\QQ$ be an elliptic curve with a point $P\in E(\QQ)$, denote $P$ as $P=\left(\frac{A}{D^{2}},\frac{B}{D^{3}}\right)$
We define {\it the height} of a point $h(P)$ by using its $x$-coordinate:
$$h(P)=\log(\max(|A|,D^{2})).$$
Moreover, we define {\it the canonical height} of a point $\hat{h}(P)$ by
$$\hat{h}(P)=\lim_{k\to\infty}\frac{h(2^{k}P)}{4^{k}}.$$
\end{Definition}

\begin{Remark}
\label{remark1}
Note that when we have $[l]P=\left(\frac{A_{l}}{D_{l}^{2}},\frac{B_{l}}{D_{l}^{3}}\right)$ with $\gcd(A_{l},D_{l})=1$, then
$$\lim_{l\to\infty}\frac{\log(D_{l}^{2})}{l^{2}}=\lim_{l\to\infty}\frac{|A_{l}|}{l^{2}}=\hat{h}(P).$$
For instance, if we choose $l=3^{k}$ for some $k$, we can also represent
$$\hat{h}(P)=\lim_{k\to\infty}\frac{\log(D_{3^{k}}^{2})}{9^{k}}=\lim_{k\to\infty}\frac{|A_{3^{k}}|}{9^{k}}.$$
This observation will be useful in Section 3.2.
\end{Remark}

Along with the above Corollary, let's recall the following theorem from \cite{BDKRW}.

\begin{Theorem}
\label{Theorem1}
Let $E/\QQ$ be an elliptic curve with a fixed point $P\in E(\QQ).$ If $\hat{h}(P)$ denotes the canonical height of $P$, we get
$$\lim_{k\to\infty}\frac{\log(F_{k}(E,P))}{4^{k}}=\frac{3}{8}\hat{h}(P)$$.
\end{Theorem}

\begin{proof}
The result follows directly from Defintion \ref{Definition2}. See \cite[Theorem 10]{BDKRW}.
\end{proof}

Using Corollary \ref{Corollary2} and Theorem \ref{Theorem1} , we can prove the Primality conjecture for magnified elliptic Fermat numbers.

\begin{Theorem}
Let $E/\QQ$ be a minimal magnified elliptic curve with a fixed point $P\in E(\QQ)$ having a (nonzero) odd-degree isogeny $\phi: E'\rightarrow E$ from a minimal elliptic curve $E'/\QQ$ satisfying $\phi(P')=P$. Then the terms $F_{k}(E,P)$ are composite for sufficiently large $k$.
\end{Theorem}

\begin{proof}
Using Siegel's Theorem, we know
$$\hat{h}(P)=m\hat{h}(P'),$$
where $m$ is the degree of the given isogeny $\phi$. Therefore, for sufficiently large $k$, there is a prime divisor which is a proper divisor of $F_{k}(E,P)$ by Theorem \ref{Theorem1}.
\end{proof}

\begin{Example}
We can see the divisibility of corresponding elliptic Fermat numbers using the following magnified elliptic divisibility sequence, which has a degree $3$ isogeny $\phi$ that maps
$$E_{1}':y^{2}=x^{3}-9x+9,~~~\text{with}~P'=[1,1]$$
to
$$E_{1}:y^{2}=x^{3}-189x-999,~~~\text{with}~P=[-8,1].$$
Then we get the following factorizations, and we can check the divisibility of corresponding elliptic Fermat numbers.\\
\\
\begin{tabular}{l|l}
\def\arraystretch{2}
{$F_{1}(E_{1}',P')=1$} & {$F_{1}(E_{1},P)=2$}\\ \\
{$F_{2}(E_{1}',P')= 17$} & {$F_{2}(E_{1}, P) = 2 * 17 * 19$}\\ \\
{$F_{3}(E_{1}',P')=53 * 127$} & {$F_{3}(E_{1},P) =2 * 53 * 127 * 10799 * 14867$}\\ \\
{$F_{4}(E_{1}',P')=89 * 179 * 307$} & {$F_{4}(E_{1}, P) =2 * 89 * 179 * 307 * 757 * 5813 *67211 $}\\
{$~~~~~~~~~~~~~~~~~~ * 5813 * 838133$} & {$~~~~~~~~~~~~~~~~~* 838133 * 265666679 * 3205176128020873$}\\

{$~~~~~~~~\vdots$} & {$~~~~~~~~\vdots$}
\def\arraystretch{1}
\end{tabular}
\\
\\
Similarly, for a degree $7$ isogeny which maps
$$E_{2}':y^{2}+xy=x^{3}-x^{2}+x+1,~~~\text{with}~Q'=[0,1]$$
to
$$E_{2}:y^{2}+xy=x^{3}-x^{2}-389x-2859,~~~\text{with}~Q=[26,51],$$
we have\\
\\
\begin{tabular}{l|l}
\def\arraystretch{2}
{$F_{1}(E_{2}',Q')=1$} & {$F_{1}(E_{2},Q)=1$}\\ \\
{$F_{2}(E_{2}',Q')=3$} & {$F_{2}(E_{2},Q) = 3 * 701$}\\ \\
{$F_{3}(E_{2}',Q')=11$} & {$F_{3}(E_{2},Q)= 11* 233 * 2887 * 273001$}\\ \\
{$F_{4}(E_{2}',Q')=1523 * 15443$} & {$F_{4}(E_{2}, Q) = 103 * 131 * 311 * 467 * 1523 * 11831 $}\\
{} & {$~~~~~~~~~~~~~~~~~~ * 15443* 12539851 * 7015932452763098743789$}\\
{$~~~~~~~~\vdots$} & {$~~~~~~~~\vdots$}
\def\arraystretch{1}
\end{tabular}

\end{Example}

\section{Generalized elliptic Fermat numbers and their properties}
\subsection{Coprimality and order universality}

The authors of \cite{BDKRW}, motivated by the coprimality of the classical Fermat numbers, show that any two distinct elliptic Fermat numbers are either relatively prime, or else their only common factor is 2. In this section, we prove an analogous result for generalized elliptic Fermat numbers generated by odd $m$. We also generalize the order universality properties stated in \cite[Theorem 4]{BDKRW} and \cite[Corollary 5]{BDKRW}.

Before proving any results, we present an example of a sequence of generalized elliptic Fermat numbers generated by $m = 3$:

\begin{Example}
\label{GEFN ex}
Let $E : y^2 = x^3 + x^2 -4x$, $P = (-2,2)$ and $m = 3$. The first four generalized elliptic Fermat numbers are listed below.\\

\begin{tabular}{l|l}
\def\arraystretch{2}
{$P = (\frac{-2}{1^2},\frac{2}{1^3})$} & {$F_0^{(3)}(E, P) = 1$}\\ \\
{$3P = (\frac{-2}{3^2},\frac{26}{3^3})$} & {$F_{1}^{(3)}(E,P) = \frac{3}{1} = 3$} \\ \\
{$9P = (\frac{-213293858}{10593^2}, \frac{2478721052834}{10593^3})$} & {$F_2^{(3)}(E, P) = \frac{10593}{3} = 3531$}\\ \\
{$27P = (\frac{-2387\ldots4098}{4777\ldots2659^2},\frac{7135\ldots8638}{4777\ldots2659^3}$)} & 
{$F_3^{(3)}(E, P) = \frac{4777\ldots2659}{10593} = 4509\ldots2163$ (33 digits)}\\
\def\arraystretch{1}
\end{tabular}

\end{Example}

We will now prove the following "coprimality" theorem, which states that the gcd of any two generalized elliptic Fermat numbers generated by an odd integer $m$ must divide $m$:

\begin{Theorem}[Coprimality]
Let $F=(F_{k}^{(m)}(E,P))_{k\in\NN}$ be the sequence of generalized elliptic Fermat numbers for a fixed elliptic curve $E$, a rational point $P\in E(\QQ)$ and an odd integer $m \geq 1$. Then for all distinct $k, \ell \geq 0$, 
$$\gcd({F_k}^{(m)}(E,P),{F_\ell}^{(m)}(E,P)) \mid m.$$
\end{Theorem}

The heart of the proof relies on the following lemma from \cite{SS}:

\begin{Lemma}
\label{SS lemma}
Let $D = (D_n)_{n \geq 1}$ be a minimal EDS, let $n \geq 1,$ and let $p$ be a prime satisfying $p \mid D_n.$

\begin{enumerate}[(a)]
\item For all $m \geq 1$ we have
$$\ord_p(D_{mn}) \geq \ord_p(mD_n).$$
\item The inequality in (a) is strict,
$$\ord_p(D_{mn}) > \ord_p(mD_n),$$
if and only if\\
$p = 2$, $2 \mid m$, $\ord_2(D_n) = 1$ and ($E$ has ordinary or multiplicative reduction at 2).
\end{enumerate}
\end{Lemma}

For our purposes, the conditions of (b) will never be met, since we are only working with odd $m$. Thus, we always have equality, i.e., if a prime $p$ satisfies $p \mid D_n$, then
$$\ord_p(D_{mn}) = \ord_p(mD_n).$$ 
We can apply Lemma \ref{SS lemma} to generalized elliptic Fermat numbers in the following way:

\begin{Proposition}
\label{ord prop}
Let $D = (D_n)_{n \geq 1}$ be a minimal EDS, and let $m \geq 1$ be an odd integer. If $p \mid D_{m^{s-1}}$ for some $s \geq 1$, then
$$\ord_p({F_s}^{(m)}) = \ord_p(m).$$
\end{Proposition}

\begin{proof}
Consider the case of Lemma \ref{SS lemma} where $n = m^{s-1}$. Then if $p \mid D_{m^{s-1}}$ for some $s \geq 1$,
$$\ord_p(D_{m^s}) = \ord_p(mD_{m^{s-1}}).$$
It immediately follows that
$$\ord_p({F_s}^{(m)})=\ord_p\left(\dfrac{D_{m^s}}{D_{m^{s-1}}}\right) = \ord_p(m).$$
\end{proof}

We now use Proposition \ref{ord prop} to prove Theorem \ref{GEFN coprimality theorem}. For simplicity, we will let ${F_k}^{(m)}$ denote ${F_k}^{(m)}(E,P)$ whenever it appears in the rest of the paper.

\begin{proof}
Let $p$ be prime, and suppose $p \mid D_{m^{s-1}}$ for some $s \geq 1$. Let $t$ be the smallest index for which $p \mid D_{m^{t-1}}$, i.e., let $t = \min\{s \geq 1 : p \mid D_{m^{s-1}}\}$. Since $\{D_n\}$ is a divisibility sequence, it is given that $p \mid D_{m^{t-1}}$ implies $p \mid D_{m^{k-1}}$ for all $k \geq t$. Without loss of generality, we assume $k<\ell$ throughout the proof. Thus, for all distinct $k,\ell$ with $\ell> k \geq t$, we can use Proposition \ref{ord prop} in order to conclude that
$$\ord_p({F_k}^{(m)}) = \ord_p({F_\ell}^{(m)}) = \ord_p(m).$$

Therefore, for each prime $p$ that divides a term in $\{D_n\}$, and for all distinct $k, \ell \geq t$, where $t$ is the entry point of $p$, we have shown that
\begin{equation}
\label{coprimality eq 1}
\ord_p(\gcd({F_k}^{(m)},{F_\ell}^{(m)})) = \ord_p(m).
\end{equation}

For all distinct $k,\ell$ with $k<t-1$, $p$ is not a factor of $\gcd({F_k}^{(m)},{F_\ell}^{(m)}))$. To see this, note that $t = \min\{s \geq 1 : p \mid D_{m^{s-1}}\}$ implies $p \nmid D_{m^{k}}.$ So $p \nmid {F_k}^{(m)},$ and for all distinct $k, \ell$ with $k < t-1$, we have the desired
\begin{equation}
\label{coprimality eq 2}
\ord_p(\gcd({F_k}^{(m)},{F_\ell}^{(m)})) = 0.
\end{equation}

When $k=t-1$ and $\ell>k$, we have
$$p\nmid D_{m^{k-1}} ~\text{and}~p\mid D_{m^{k}}$$
and
$$p\mid D_{m^{\ell-1}}~\text{and}~p\mid D_{m^{\ell}}.$$
Therefore, we have $\ord_{p}(F_{k}^{(m)})>0$ and $\ord_{p}(F_{\ell}^{(m)})=\ord_{p}(m)$ and we get
\begin{equation}
\label{coprimality eq 3}
\ord_{p}(\gcd({F_k}^{(m)},{F_\ell}^{(m)}))= \ord_{p}(m).
\end{equation}
It follows from Equations \eqref{coprimality eq 1}, \eqref{coprimality eq 2}, and \eqref{coprimality eq 3} that for any distinct $k, \ell$ and any prime $p$,

\begin{equation}
\label{gcd}
\ord_p(\gcd({F_k}^{(m)},{F_\ell}^{(m)})) = 
\begin{dcases}
\ord_p(m) & \text{if } p \mid D_{m^{t-1}} \text{ for some } t \leq k\\
0 & \text{ otherwise.}
\end{dcases}
\end{equation}
This implies
$$\gcd({F_k}^{(m)},{F_\ell}^{(m)}) \mid m.$$
\end{proof}

\begin{Remark}
\label{coprimality rmk}
The proof of Theorem \ref{GEFN coprimality theorem} actually implies a more specific result than the theorem states. For once a prime $p$ appears as a divisor of some ${F_t}^{(m)}$, then $p^{\ord_p(m)} \mid {F_k}^{(m)}$ for all $k \geq t$. Thus, if for example $3 \mid {F_t}^{(15)}(E,P)$, then Theorem \ref{GEFN coprimality theorem} only tells us that $\gcd({F_k}^{(15)},{F_\ell}^{(15)}) \in \{1, 3, 5, 15\}$ for all distinct $k, \ell \geq t$, but in actuality, we know that $\gcd({F_k}^{(15)},{F_\ell}^{(15)}) \in \{3, 15\}$ because $3 \nmid 1$ and $3 \nmid 5$. This is even stronger in the case where $m = p^a$ for some prime $p$, as stated in the proof of the corollary below.
\end{Remark}

\begin{Corollary}
\label{GEFN coprimality cor}
Let $F=(F_{k}^{(p^a)}(E,P))_{k\in\NN}$ be the sequence of generalized elliptic Fermat numbers for a fixed elliptic curve $E$, a rational point $P\in E(\QQ)$ and an odd prime power $p^a$. Then for all distinct $k, \ell \geq 0$, 
$$\gcd({F_k}^{(p^a)}(E,P),{F_\ell}^{(p^a)}(E,P)) \in \{1, p^a\}.$$


\end{Corollary}

\begin{proof}
The proof follows from \eqref{gcd}.






\end{proof}

\begin{Example}
\label{Coprimality ex}
We factor the generalized elliptic Fermat sequence from Example \ref{GEFN ex}:\\

\begin{tabular}{l}
{$F_0^{(3)}(E, P) = 1$}\\
\\
{$F_{1}^{(3)}(E,P) = \frac{3}{1} = 3$}\\
\\
{$F_2^{(3)}(E, P) = \frac{10593}{3} = 3531 = 3 * 11 * 107$}\\
\\
{$F_3^{(3)}(E, P) = \frac{4777\ldots2659}{10593} = 4509\ldots2163 = 3 * 3240769000879427 * 46385324158085723$}\\
\\
\end{tabular}\\

\end{Example}

We see that $\gcd({F_0}^{(3)},{F_1}^{(3)}) = 1$, while $\gcd({F_k}^{(3)},{F_\ell}^{(3)}) = 3$ for distinct $1 \leq k, \ell \leq 3$.

In addition to proving the quasi-coprimality of elliptic Fermat numbers, the authors of \cite{BDKRW} include a result connecting divisibility with order, which they call \emph{order universality}. This property holds in full force for generalized elliptic Fermat numbers. In fact, the proofs are nearly the same as the proofs for the case where $m = 2$.

\begin{Theorem}
\label{GEFN OU}
Let $F=(F_{k}^{(m)})_{k\in\NN}$ be the sequence of generalized elliptic Fermat numbers for a fixed elliptic curve $E$, a rational point $P\in E(\QQ)$ and an odd integer $m \geq 1$. Then for all $N \in \NN$ satisfying $\gcd(N, 6\Delta(E)) = 1$,
$$P \textrm{ has order } m^k \textrm{ in } E(\ZZ/N\ZZ) \Longleftrightarrow N \mid {F_0}^{(m)} \cdots {F_k}^{(m)} \textrm{ and } N \nmid {F_0}^{(m)}\cdots {F_{k-1}}^{(m)}.$$ 
\end{Theorem}

\begin{proof}
The proof is identical to that of Theorem 4 in \cite{BDKRW}. Namely, we define a homomorphism $\phi: E(\QQ) \rightarrow E(\ZZ/N\ZZ)$ that maps $P \mapsto P \bmod(n)$, then use the fact that $\phi(p^kP) = p^k\phi(P)$ to demonstrate that $P$ has order $m^k$ in $E(\ZZ/N\ZZ)$ exactly when $N \mid {F_0}^{(m)} \cdots {F_k}^{(m)} $ and $ N \nmid {F_0}^{(m)} \cdots {F_{k-1}}^{(m)}.$ 
\end{proof}

\begin{Corollary}
\label{GEFN OU cor}
Let $F=(F_{k}^{(m)})_{k\in\NN}$ be the sequence of generalized elliptic Fermat numbers for a fixed elliptic curve $E$, a rational point $P\in E(\QQ)$ and an odd integer $m \geq 1$. Let $p$ be an odd prime. Then 
$$P \textrm{ has order } m^k \textrm{ in } E(\FF_p) \Longleftrightarrow p \mid {F_k}^{(m)}.$$
\end{Corollary}

\begin{proof}
The proof is similar to that of Corollary 5 in \cite{BDKRW}. However, whereas the proof in \cite{BDKRW} relies on the fact that $\gcd({F_k}^{(2)},{F_\ell}^{(2)}) \in \{1,2\}$, here we require that $p \nmid m$ in order to make use of Theorem \ref{GEFN coprimality theorem}, which tells us that $\gcd({F_k}^{(m)},{F_\ell}^{(m)}) \mid m.$

The adapted proof proceeds as follows:

If $p \mid {F_0}^{(m)}(E, P) \cdots {F_k}^{(m)}(E, P)$ and $p \nmid {F_0}^{(m)}(E, P) \cdots {F_{k-1}}^{(m)}(E, P)$, then $p \mid {F_k}^{(m)}(E, P)$. Conversely, if $p \mid {F_k}^{(m)}(E, P)$, then $p \mid {F_0}^{(m)}(E, P) \cdots {F_k}^{(m)}(E, P)$. Theorem \ref{GEFN coprimality theorem} gives us $p \nmid {F_i}^{(m)}(E, P)$ for all $i \neq k$, implying $p \nmid {F_0}^{(m)}(E, P) \cdots {F_{k-1}}^{(m)}(E, P)$. Thus we have shown that $p \mid {F_k}^{(m)}(E, P)$ if and only if $p \mid {F_0}^{(m)}(E, P) \cdots {F_k}^n(E, P)$ and $p \nmid {F_0}^{(m)}(E, P) \cdots {F_{k-1}}^{(m)}(E, P)$, and the desired result follows from Theorem \ref{GEFN OU}.

\end{proof}

\begin{Example}
Using the curve $E$, point $P$ and integer $m$ from Example \ref{GEFN ex}, observe that the order of $P \in E(\FF_{593}) = 3^2$, and indeed, $593 \mid {F_2}^{(3)} = 3 * 593 = 1779.$

\end{Example}

\subsection{Primality conjecture for generalized elliptic Fermat numbers}
We can also extend previous discussions about the Primality of elliptic Fermat numbers to generalized elliptic Fermat numbers. First, we state an analogous conjecture for generalized elliptic Fermat numbers.

\begin{Conjecture}[Primality Conjecture for generalized elliptic Fermat numbers] 
Let $F=(F_{k}^{(m)}(E,P))_{k\in\NN}$ be the sequence of generalized elliptic Fermat numbers for a fixed elliptic curve $E$ and a rational point $P\in E(\QQ)$. Then $F$ contains only finitely many prime terms.
\end{Conjecture}

Using Corollary \ref{GEFN OU cor}, we can prove the following result.

\begin{Lemma}
Let $E/\QQ$ be a minimal magnified elliptic curve with a fixed point $P\in E(\QQ)$ having a (nonzero) degree $d$ isogeny $\phi: E'\rightarrow E$ from a minimal elliptic curve $E'/\QQ$ satisfying $\phi(P')=P$. For sufficiently large $k$ and $m$ with $\gcd(m,d)=1$, we have
$$\gcd(F_{k}^{(m)}(E',P'),F_{k}^{(m)}(E,P))\neq 1.$$
\end{Lemma}

\begin{proof}
Let $p$ be a fixed prime which divides $F_{k}^{(m)}(E',P')$. Let $S$ be the set of primes for which $E$, $E'$, and $\phi$ cannot give an isogeny modulo $p$. Note that $S$ is a finite set. Without loss of generality, we can assume $p\notin S$. From Corollary \ref{Cor1}, it is sufficient to prove that the isogeny $\phi$ reduction modulo $p$ (which is again an isogeny) preserves the order $m^{k}$ of $P'$
$$\phi:E'(\FF_{p})\rightarrow E(\FF_{p}).$$
We consider the dual isogeny $\hat{\phi}$ of $\phi$. We know
$$\hat{\phi}\circ \phi=[d]$$
where $[d]$ is the multiplication-by-$d$ map on $E'$. Then the map
$$[d]: E'(\FF_{p})\xrightarrow{\phi} E(\FF_{p})\xrightarrow{\hat{\phi}} E'(\FF_{p})$$
preserves the order of $m^{k}$ of $P'$ and so does $\phi$.    
\end{proof}  

\begin{Proposition}
Let $E/\QQ$ be a minimal magnified elliptic curve with a fixed point $P\in E(\QQ)$ having a (nonzero) degree $d$ isogeny $\phi: E'\rightarrow E$ from a minimal elliptic curve $E'/\QQ$ satisfying $\phi(P')=P$. For sufficiently large $k$, we actually have the following divisibility: For $m$ relatively prime to $d$,
$$F_{k}^{(m)}(E',P')\mid F_{k}^{(m)}(E,P).$$
\end{Proposition}

\begin{Example}
We can see the divisibility of corresponding generalized elliptic Fermat numbers using the following magnified elliptic divisibility sequence, which has a degree $2$ isogeny $\phi$ that maps
$$E': y^{2}=x^{3}+x^{2}-4x,~~~\text{with}~P'=[-2,2]$$
to
$$E: y^{2}=x^{3}+x^{2}+16x+16~~~\text{with}~P=[0,4],$$
then we get following list of $F_{k}^{(3)}(E',P')$ and $F_{k}^{(3)}(E,P)$.
\\
\\
\begin{tabular}{l|l}
\def\arraystretch{2}
{$F_{1}^{(3)}(E',P')=3$} & {$F_{1}^{(3)}(E,P)=3$}\\ \\
{$F_{2}^{(3)}(E',P')= 3 * 11 * 107$} & {$F_{2}^{(3)}(E, P) =  3 * 11 * 23 * 107 * 449$}\\ \\
{$F_{3}^{(3)}(E',P')=3 * 3240769000879427$} & {$F_{3}^{(3)}(E,P) =3 * 114078700999 * 3240769000879427 $}\\
{$~~~~~~~~~~~~~~~~~~~~~ * 46385324158085723$} & {$~~~~~~~~~~~~~~~~~~~* 46385324158085723 * 927508107491526089159$}\\
{$~~~~~~~~\vdots$} & {$~~~~~~~~\vdots$}
\def\arraystretch{1}
\end{tabular}
\\
\end{Example}
From Remark \ref{remark1}, we can also describe the growth of generalized elliptic Fermat numbers using the canonical height of $P$.

\begin{Theorem}
\label{Theorem2}
Let $E/\QQ$ be an elliptic curve with a fixed point $P\in E(\QQ)$. Denote by $\hat{h}(P)$ the canonical height of $P$. For any $m$, we get
$$\lim_{k\to\infty}\frac{\log(F_{k}^{(m)}(E,P))}{m^{2k}}=\left(\frac{1}{2}-\frac{1}{2m^{2}}\right)\cdot \hat{h}(P).$$
\end{Theorem}  

\begin{proof}
\begin{align*}
\lim_{k\to\infty}\frac{\log(F_{k}^{(m)}(E,P))}{m^{2k}} &=\lim_{k\to\infty}\frac{\log\left(\frac{D_{m^k}}{D_{m^{k-1}}}\right)}{m^{2k}}\\
&=\lim_{k\to\infty}\frac{\log(D_{m^{k}})}{m^{2k}}-\lim_{k\to\infty}\frac{\log(D_{m^{k-1}})}{m^{2k}}\\
&=\lim_{k\to\infty}\frac{1}{2}\cdot \frac{\log(D^{2}_{m^{k}})}{m^{2k}}-\lim_{k\to\infty}\frac{1}{2m^{2}}\cdot \frac{\log(D^{2}_{m^{k-1}})}{m^{2(k-1)}}\\
&=\left(\frac{1}{2}-\frac{1}{2m^{2}}\right)\cdot \hat{h}(P).
\end{align*}
\end{proof}
Note that when $m=2$, the result coincides with Theorem \ref{Theorem1}.

\begin{Theorem}
Let $E/\QQ$ be a minimal magnified elliptic curve with a fixed point $P\in E(\QQ)$ having a (nonzero) degree $d$ isogeny $\phi: E'\rightarrow E$ from a minimal elliptic curve $E'/\QQ$ satisfying $\phi(P')=P$. For $m$ relatively prime to $d$, $F^{(m)}_{k}(E,P)$ are composite for sufficiently large $k$.
\end{Theorem}

\begin{proof}
Using Siegel's Theorem, we know
$$\hat{h}(P)=d\hat{h}(P'),$$
where $d$ is the degree of the given isogeny $\phi$. Therefore, for sufficiently large $k$, there is a prime divisor which is a proper divisor of $F^{(m)}_{k}(E,P)$ by Theorem \ref{Theorem2}.
\end{proof}

\bibliographystyle{plain}      
\bibliography{edsref.bib}

\end{document}